\documentclass{amsart}

\usepackage{amsmath,amsfonts,amssymb,amsthm}
\usepackage{mathrsfs}
\usepackage[shortlabels]{enumitem}
\usepackage{hyperref}
\hypersetup{colorlinks=true, citecolor=black, linkcolor=black, urlcolor=black, filecolor=black, bookmarksopen=true, pdfpagemode=FullScreen,
pdftitle={A Small Radon-Nikodym Compact Space From a Parametrized Diamond},pdfauthor={Arturo Martínez-Celis, Adam Morawski}}

\newtheorem{theorem}{Theorem}[section]

\newtheorem{lemma}[theorem]{Lemma}
\newtheorem{proposition}[theorem]{Proposition}

\newtheorem{fact}[theorem]{Fact}

\newtheorem{question}{Question}

\newtheorem*{maintheorem}{Theorem \ref{maintheorem}}

\theoremstyle{remark}
\newtheorem*{remark}{Remark}

\theoremstyle{definition}
\newtheorem{definition}[theorem]{Definition}
\newtheorem*{acknowledgements}{Acknowledgments}

\newcommand{\U}{\mathcal{U}}
\newcommand{\M}{\mathcal{M}}

\newcommand{\A}{\mathcal{A}}
\newcommand{\LL}{\mathbb{L}}

\newcommand{\non}[1]{\mathrm{non}(#1)}

\newcommand{\baire}{\omega^\omega}

\newcommand{\cantor}{{(2^{<\omega})}^{2^{<\omega}}}
\newcommand{\cantortree}{{2^{<\omega}}}

\title{A small Radon-Nikod\'{y}m compact space from a parametrized diamond}

\author{Arturo Mart\'{i}nez-Celis}
\address{Instytut Matematyczny, Uniwersytet Wroc\l awski, pl. Grunwaldzki 2, 50-384 Wrocław, Poland}
\email{\href{mailto:arturo.martinez-celis@math.uni.wroc.pl}{arturo.martinez-celis@math.uni.wroc.pl}}

\author{Adam Morawski}
\address{Instytut Matematyczny, Uniwersytet Wroc\l awski, pl. Grunwaldzki 2, 50-384 Wrocław, Poland}
\address{Univerzita Karlova, Matematicko-fyzik\'{a}ln\'{i} fakulta,
Ke Karlovu 3, 121 16 Praha 2, Czech Republic}
\email{\href{mailto:morawski@math.cas.cz}{morawski@math.cas.cz}}

\begin{document}

\begin{abstract}
A compact space $K$ is \emph{Radon-Nikod\'{y}m} if there is a lower semi-continuous metric fragmenting $K$. In this note, we show that, under $\diamondsuit (\non{\M})$, there is a Radon-Nikod\'{y}m compact space of weight $\aleph_1$ with a continuous image that is not Radon-Nikod\'{y}m, which partially answers a question posed in \cite{AvilesKoszmider}.
\end{abstract}

\maketitle
\section{Introduction.}
The study of Radon-Nikod\'{y}m compact spaces can be traced back to \cite{GSGReynov}, where they are introduced as a generalization of Eberlein compact spaces. These spaces were originally defined as weak compact subspaces $K \subseteq X^*$, where every Radon probability measure on $K$ is supported on a countable union of strongly compact sets. This definition has been studied extensively, and several connections have been found in Banach space theory and topology. A summary of some of these can be found in \cite{OverclassesFabian}.

Recall that metric $d$ \emph{fragments} $K$ if for every $\varepsilon>0$ and every non-empty $C \subseteq K$, there is an open set $U \subseteq K$ such that $U \cap C \neq \emptyset$ and $\mathrm{diam}(U\cap C) < \varepsilon$, and $d$ is \emph{lower semi-continuous} if it is lower semi-continuous as a function on $K \times K$. In \cite{NamiokaRNFragment}, Namioka proved that a compact space $K$ is Radon-Nikod\'{y}m if and only if there is a lower semi-continuous metric that fragments $K$. In the same work, the author asked whether the class of Radon-Nikod\'{y}m compact spaces is closed under continuous functions. This question will be the central point of our work.

In \cite[p. 104]{GenTopII}, a weaker notion of fragmentability is considered, attributed to Reznichenko: A compact space $K$ is \emph{quasi Radon-Nikod\'{y}m} (originally called strongly fragmentable) if there is a metric $d$ that fragments $K$, and for every different $x,y \in K$, there are open sets $U,V$ such that $x \in U, y \in V$ and $d(U,V)>0$. Clearly all Radon-Nikod\'{y}m compact spaces are quasi Radon-Nikod\'{y}m compact spaces, and, by a result of Namioka \cite{GenerRNSpaces}, the class of quasi Radon-Nikod\'{y}m compact spaces is closed under continuous images.

Several results relating quasi Radon-Nikod\'{y}m and Radon-Nikod\'{y}m compact spaces can be found in the literature: For example, Arvanitakis \cite{ArvanitakisRNComp} proved that totally disconnected quasi Radon-Nikod\'{y}m compact spaces are Radon-Nikod\'{y}m, and Avilés \cite{AvilesRNUnderB} proved that quasi Radon-Nikod\'{y}m spaces of weight less than $\mathfrak{b}$ are Radon-Nikod\'{y}m. The reader can consult \cite{SurveyProblem} for other related questions and results.

In 2013, Avilés and Koszmider answered Namioka's question in the negative.

\begin{theorem}[\cite{AvilesKoszmider}]
    There are compact spaces $\mathbb{L}_0, \mathbb{L}_1$ of weight $\mathfrak{c}$ and a continuous surjection $\pi: \mathbb{L}_0 \rightarrow  \mathbb{L}_1$ such that $\mathbb{L}_0$ is a Radon-Nikod\'{y}m compact space, and $\mathbb{L}_1$ is not.
\end{theorem}

In this work, Avilés and Koszmider asked about the possible weights of such spaces. In particular, they asked if $\mathfrak{b}$ is the smallest weight of those spaces, since, by Avilés' theorem of \cite{AvilesRNUnderB}, the weight of such spaces cannot be smaller than $\mathfrak{b}$, thus, under $\mathfrak{b} = \mathfrak{c}$, there are no such spaces of weight less than $\mathfrak{c}$. The purpose of this note is to give a construction of such spaces of weight $\aleph_1$, under the guessing principle $\diamondsuit(\non{\mathcal{M}})$, which is consistent with the negation of the continuum hypothesis.

The guessing principles that we use in this work can be described using Borel relational systems. These relational systems were originally considered by Vojt\'a\v{s} in \cite{InvariantsGaloisTukey}, and have been used as a tool to describe and compare many cardinal characteristics (for example, see \cite{BlassCardinalCharacteristics}): A \emph{Borel relational system} is a triple $\mathbf{A} = \langle A_-,A_+,A\rangle$ where $A_-$ and $A_+$ are Borel subsets of a Polish space and $A \subseteq A_- \times A_+$ is a Borel relation. The \emph{norm} of a relational system $\mathbf{A} = \langle A_-,A_+,A\rangle$ is the smallest cardinality of a set $D \subseteq A_+$ such that for every $a \in A_-$ there is a $d \in D$ such that $\langle a,d \rangle \in A$. Many of the most well-known cardinal invariants of the continuum can be stated as the norm of a relational system: For example, $\non{\M}$ can be stated as the norm of $\mathbf{M} = \langle \M, \baire, \not\ni \rangle$, where $\M$ is the collection of all $F_\sigma$ meager sets of $\baire$, and $\mathfrak{b}$ can be stated as the norm of $\mathbf{B} = \langle \baire, \baire, \not\leq^* \rangle$. The reader can easily verify that both $\mathbf{M}$ and $\mathbf{B}$ are Borel relational systems.

In \cite{ParametrizedDiamonds}, the authors introduced a plethora of combinatorial principles similar to $\diamondsuit$, called parametrized diamond principles: Given a Borel relational system $\mathbf{A} = \langle A_-,A_+,A \rangle$, the principle $\diamondsuit (\mathbf{A} )$ is the statement that for every Borel $F : 2^{<\omega_1} \rightarrow \A_-$ there is a function $g : \omega_1 \rightarrow A_+$ such that, for any $\gamma \in 2^{\omega_1}$, the set $\{ \alpha \in \omega_1 : \langle g(\alpha), F(\gamma \restriction \alpha)\rangle \in A \}$ is stationary in $\omega_1$. One can easily verify that $\diamondsuit (\mathbf{A})$ is a consequence of $\diamondsuit$ and that $\diamondsuit (\mathbf{A})$ implies that the norm of $\mathbf{A}$ is at most $\aleph_1$. Here, a function $F$ with domain $2^{<\omega_1}$ is \emph{Borel} if for each $\alpha < \omega_1$, the restrictions $F \restriction {2^\alpha}$ are Borel functions between Polish spaces.

These principles have been used extensively, not only for set-theoretical purposes, but also in topology. Some applications of different kinds of parametrized diamonds can be found in \cite{ScarStone}, \cite{irresolvablespace}, \cite{DowkerSpace}, \cite{CoveringProperties}.

In our work, we will be using mostly $\diamondsuit(\mathbf{M})$, which is known as $\diamondsuit (\non{\M})$, and we will be using it in the following form:
\begin{definition}
    The \emph{diamond of }$\non{\M}$, denoted by $\diamondsuit (\non{\M})$ is the following combinatorial principle:

    For every Borel $F: 2^{<\omega_1} \rightarrow \M$ there is a sequence $\Delta : \omega_1 \rightarrow \cantor \times \omega^{\cantortree \times \omega}$ such that for any $\gamma \in 2^{\omega_1}$, the set $\{ \alpha \in \omega_1 : \Delta(\alpha) \notin F(\gamma \restriction \alpha) \}$ is stationary in $\omega_1$.
\end{definition}

The relationship between different parametrized diamonds and their consistency is widely discussed in \cite{ParametrizedDiamonds}. In particular, $\diamondsuit(\non{\mathcal{M}})$ follows from the classical $\diamondsuit$, but it also holds in Sacks' or Millers' model (see \cite[Chapter 7]{BarJu}), as both Sacks and Miller forcing fall into the scope of \cite[Theorem 6.6.]{ParametrizedDiamonds}. This principle is also consistent with arbitrarily large continuum, by \cite[Theorem 6.1.]{ParametrizedDiamonds} and the fact that the classical forcing to add $\kappa$ Cohen reals $\mathbb{C}_\kappa$ is forcing equivalent to the iteration $\mathbb{C}_\kappa \ast \mathbb{C}_{\omega_1}$.

The purpose of this note is to prove the following theorem.

\begin{maintheorem}
Under $\diamondsuit (\non{\M})$ there is a Radon-Nikod\'{y}m compact space of weight $\aleph_1$, and a continuous function whose image is not Radon-Nikod\'{y}m.
\end{maintheorem}

Therefore, it is consistent to have a Radon-Nikod\'{y}m compact space of small weight, with a continuous image that is not Radon-Nikod\'{y}m, for example, in any of the models mentioned above.

This paper is organized in the following way: Section 2 is devoted to briefly introduce the notion of a basic$^*$ space, and some results that were originally proven in \cite{AvilesKoszmider} which will be needed to explain why the existence of a basic$^*$ space implies our main theorem. Section 3 is the core of our construction and is completely dedicated to construct a basic$^*$ space from $\diamondsuit (\non{\M})$, finishing in a short discussion about the limitations of this method and some possible improvements.

Our notation is standard and mostly follows \cite{TopicsBanach} and \cite{STJech}: If $X$ is a set, then $|X|$ denotes the cardinality of $X$. The ordinal $\omega = \{ 0,1,2, \ldots \}$ is the first infinite cardinal, $\omega_1$ is the first uncountable ordinal, and $\mathfrak{c}$ is the cardinality of $\mathbb{R}$. If $X,Y$ are sets, then $X^Y$ is the collection of all functions $f: Y \rightarrow X$. In particular, if $n\in \omega$, then $2^n$ is the collection of all functions $s : n \rightarrow 2 =\{ 0,1\}$, and $2^\omega = \bigcup_{n \in \omega} 2^n$. If $f$ is a function, and $A \subseteq \mathrm{dom}(f)$, then $f \restriction A$ is the usual restriction of $f$ to the set $A$. If $X$ is a topological space, then $w(X)$ is the \emph{weight} of $X$, that is, the smallest cardinality of a basis for the topology of $X$. If $s,t \in \cantortree$, then $s ^\frown t$ denotes the usual concatenation of sequences. The quantifiers $\forall^\infty n \in \omega$ and $\exists^\infty n \in \omega$ denote the quantifiers representing "for every, but finitely many $n \in \omega$", and "there is an infinite amount of $n\in \omega$", respectively.

\section{Constructing Radon-Nikod\'{y}m compact spaces from a basic* space.}

In \cite{AvilesKoszmider}, the authors construct a Radon-Nikod\'{y}m space $\mathbb{L}_0$ and a continuous function $\pi$ such that $\pi[\mathbb{L}_0] = \mathbb{L}_1$ is not Radon-Nikod\'{y}m. In that work, both $\mathbb{L}_0$ and $\mathbb{L}_1$ are constructed following a recipe that requires what they call a basic space as an ingredient. It turns out that the existence of such basic space of weight $\aleph_1$ implies the continuum hypothesis. Here, we will follow the same recipe, modifying the notion of basic space, allowing us to work without the need of the continuum hypothesis. The main purpose of this section is to show that this modification is tame enough to allow us to follow the recipe with minimal modifications, and to see that the spaces have the proper weight.

\begin{definition}
    A topological, Hausdorff space $K = A \cup B \cup \{c\}$ is \emph{pre-basic} if
    \begin{enumerate}
        \item $A = \dot{\bigcup}_{t \in \cantortree} A_t$ and each $A_t$ and $B$ are infinite sets,
        \item $A$ consists of isolated points,
        \item \label{c-set}For each $b \in B$ there is a countable set $C_b \subseteq A$ such that $\{ (C_b\setminus F) \cup \{b\}: F \text{ finite}\}$ is a local basis of $b$,
        \item $K$ is the one-point compactification of $A \cup B$.
    \end{enumerate}
Additionally, if 
\begin{enumerate}[resume]
    \item \label{basic}There is a function $D : B \rightarrow \cantor$ such that for every function $s: A \rightarrow \cantortree$ there is $b\in B$ such that there are infinitely many $n \in \omega$ such that for every $t \in 2^n$ the set $\{ a \in A_t \cap C_b : D(b)(t) = s(a) \}$ is infinite.
\end{enumerate}
then $K$ is a \emph{basic}$^*$ space.
\end{definition}

Clearly, the first three conditions imply that the space is locally compact, so condition (4) is attainable. In fact, we have the following easy properties:

\begin{fact}
    Assume that $A \cup B$ satisfies conditions $(1)$,$(2)$ and $(3)$. Then,
    \begin{itemize}
        \item for every $b,d \in B$ such that $b \neq d$, we have that $C_b \cap C_d$ is finite,
        \item for all $b \in B$, $C_b \cup \{b\}$ is compact. Moreover, $\overline{C_b}=C_b\cup\{b\}$ is closed and open in both $A\cup B$ and $K$.
    \end{itemize}
\end{fact}
\begin{proof}
    The first point follows by the fact that the space is Hausdorff and condition $(3)$. The second one follows easily from $(2)$ and $(3)$.
\end{proof}


An easy consequence of this is that $K$ is a zero-dimensional space. The weight of these spaces is also easy to calculate.

\begin{lemma}
    If $K$ is a pre-basic space, then $w(K)=|K|$. \label{lemmaweightK}
\end{lemma}
\begin{proof}
     Since both A and B are discrete in their respective subspace topologies, we get
     $$w(K)\geqslant \max\{w(A),w(B)\} = \max\{|A|,|B|\}= |K|.$$
     
     For the other inequality, we only have to show that $A\cup B$ has a basis of size $|K|$. Notice that any set that contains an infinite subset of $B$ is not compact in $A\cup B$ and therefore any compact set in $A\cup B$ is contained in a set of the form $B_0\cup\{C_b:b\in B_0\}\cup A_0$ for some finite $B_0\subseteq B,A_0\subseteq A$. Thus $\{K\setminus (B_0\cup\{C_b:b\in B_0\}\cup A_0):B_0\subseteq B,A_0\subseteq A\text{ finite}\}$ is a basis of neighborhoods $c$ of size $|K|$ (the basis for other points is given by $(2)$ and $(3)$).
\end{proof}

The only difference between a basic* space, and the notion of basic space defined in \cite[Section 3]{AvilesKoszmider}, is the condition (\ref{basic}). In that work, they instead ask for the following property:

\noindent \cite{AvilesKoszmider} There exists a function $\Psi: B \rightarrow \omega^{\cantortree}$ such that, given any family $\{ X_m^t : m \in \omega, t \in \cantortree \}$ of subsets such that, for every $t \in \cantortree$ $A_t = \bigcup_{m \in \omega} X^t_m$, there is an $\alpha \in B$ such that, for all $t \in \cantortree$, $C_\alpha \cap X^t_{\Psi(\alpha)(t)}$ is infinite.

Observe that $\Psi$ necessarily has to be onto: If $f \in \omega^{\cantortree}$, and if we consider the family $\{ X(f)^t_m : t \in \cantortree, m \in \omega \}$ where $X(f)^t_{f(t)} = A_t$, and $X(f)^t_m = \emptyset$ for every other $m \neq f(t)$, then there must be $b \in B$ such that, for all $t \in \cantortree$, $C_b \cap X(f)^t_{\Psi(b)(t)}$ is infinite. But then $C_b \subseteq X(f)^t_{\Psi(b)(t)}$ and therefore $f(t) = \Psi(b)(t)$ for all $t \in \cantortree$. This means that it is impossible to construct basic spaces of cardinality smaller than $\mathfrak{c}$.

The next step of the recipe is described as follows: Given $K=A\cup B\cup \{c\}$ a pre-basic space, let 
\[L = (A\times 2^\omega)\cup B\cup\{c\}.\]
To define the topology on $L$, pick $\U$ a basis of $K$. Then, the collection
\[
\{\{a\}\times [s]: a\in A, s\in \cantortree\} \cup \{ ((U\cap A)\times 2^\omega) \cup U\setminus A : U\in \U\}
\]
forms a basis of the topology of $L$. 

\begin{lemma} \label{lemmaweightL}
   $L$ is a zero-dimensional compact space and $w(L) = |K|.$
\end{lemma}
\begin{proof}
    It is routine to check that $L$ is the one-point compactification of $(A\times 2^\omega)\cup B$, and that it is zero-dimensional. For the other part, notice that the size of the basis of $L$ only depends on the size of $\mathcal{U}$, so it follows directly from Lemma \ref{lemmaweightK}.
\end{proof}

For the next step of the recipe, we need to recall the lexicographic order: If $x,y \in 2^\omega \cup \cantortree$ and $x \neq y$, then $x <_{lex} y$ if $x(i) = 0, y(i) = 1$ and $i = \min \{ n \in \omega : x(n) \neq y(n)\}$.

For $s,t\in \cantortree$, define $\Gamma_t^s:2^{\omega}\to 2^{\omega}$ as
\[\Gamma_t^s(\sigma)=\begin{cases}
    t^\frown (0,0,\ldots) & \text{ if } \sigma <_{lex} s^\frown (0,0,\ldots),\\
    t^\frown \lambda & \text{ if } \sigma = s^\frown\lambda,\\
    t^\frown (1,1,\ldots) & \text{ if } \sigma >_{lex} s^\frown (1,1,\ldots).
\end{cases}\]
Note that $\Gamma_t^s$ is a continuous function. Let $q:2^\omega \to [0,1]$ be the standard continuous surjection given by calculating the value of a sequence as a base 2 number, that is, $q(\tau)=\sum_{i\in x^{-1}(\omega)}2^{-(i+1)}$. From now on, assume that we have a pre-basic space $K = A \cup B \cup \{c\}$ and some arbitrary function $D : B \rightarrow \cantor$:

For $b\in B$ define $g_b:(L\setminus\{b\}) \to 2^{\omega}$ as follows:
\begin{itemize}
    \item $g_b(x) = (0,0,0,\ldots)$ if $x\not\in C_b\times2^{\omega},x\neq b$,
    \item $g_b(a,\sigma) = \Gamma_t^{D(b)(t)}(\sigma)$ for $a\in A_t\cap C_b,\sigma\in 2^\omega$,
\end{itemize}
and let $h_b:(L\setminus\{b\}) \to [0,1]$ be defined by $h_b=q\circ g_b$.

Let 
\begin{gather*}
\LL_0 = \{(u,v) \in L \times (2^\omega)^B : \forall_{b\in B} \text{ if } b\neq u \text{ then } v(b) = g_b(u)\},\\
\LL_1 = \{(u,v) \in L \times [0,1]^B : \forall_{b\in B} \text{ if } b\neq u \text{ then } v(b) = h_b(u)\}
\end{gather*}
and equip $\LL_0,\LL_1$ with the subspace topology derived from $L\times (2^\omega)^B$ and $L\times [0,1]^B$, respectively.
Clearly, $\LL_1$ is an image of $\LL_0$, by the continuous function $\pi$ that applies $q$ on each coordinate of $(2^\omega)^B$.
\begin{lemma}
    $w(\LL_0)=w(\LL_1)=|K|.$ \label{lemmaweightfinalspaces}
\end{lemma}
\begin{proof}
    It follows from Lemma \ref{lemmaweightL}, and the fact that $w((2^\omega)^B) = w([0,1]^B) = |B|$.
\end{proof}

To finish the recipe, we only need to show that $\LL_0$ is Radon-Nikod\'{y}m and that $\LL_1$ is not Radon-Nikod\'{y}m.

\begin{theorem}
    If $K$ is pre-basic, and $D$ is arbitrary, then $\LL_0$ is a Radon-Nikod\'{y}m compact space.\label{l0}
\end{theorem}
\begin{proof}
    It is immediate to see that $\LL_0$ is zero-dimensional, so by \cite[Theorem 3.6]{ArvanitakisRNComp} we only need to check that it is quasi Radon-Nikod\'{y}m. The proof is exactly the same as the proof of \cite[Prop. 4.1]{AvilesKoszmider} as it depends only on the continuity of the functions $\Gamma^s_t$.
\end{proof}
\begin{theorem}
   If $K$ is basic* and $D$ witnesses that $K$ satisfies the condition (\ref{basic}), then $\LL_1$ is not a Radon-Nikod\'{y}m compact space.\label{l1}
\end{theorem}
\begin{proof}
    The proof is almost identical to \cite[Prop. 4.3]{AvilesKoszmider}, with a small modification to take into account the different property of the function $D$. First, notice that, for $u \in L$ such that $ u = \langle a,x \rangle \in A \times 2^\omega $, there is a single function $v \in [0,1]^B$ (in this case, the function $v(b) = h_b(u)$) such that $\langle u,v \rangle \in \LL_1$. Such $\langle u,v \rangle$ will be denoted as $a + x$.  Let $\vec{0}$ and $\vec{1}$ be the constant functions $0$ and $1$, respectively.

    Aiming towards a contradiction, assume that $\LL_1$ is Radon-Nikod\'{y}m, so there must be a lower semi-continuous metric $\delta$ fragmenting $\LL_1$. For each $t \in \cantortree$ and each $a \in A_t$, consider the set $\{ a + x : x \in [t] \}$, and apply the fragmentability of $\LL_1$ to this set to find $s(a) \in \cantortree$ extending $t$ such that 

\begin{equation}
\tag{*}\label{smldm}\delta (a + s(a)^\frown\vec{0}, a + s(a)^\frown\vec{1}) < \frac{1}{4^{|t|}}.
\end{equation}

Then, according to condition (\ref{basic}), there must be some $b \in B$ such that, there is an infinite set $W \subseteq \omega$ such that for every $n \in W$ and every $t \in 2^n$, the set $\{ a \in A_t \cap C_b : D(b)(t) = s(a) \}$ is infinite. Enumerate this set as $\{ a^t_k : k \in \omega \}$.

For every $\xi \in [0,1]$ there is a unique function $v : B \rightarrow [0,1]$ such that $v(b) = \xi$ and $\langle b, v \rangle \in \LL_1$ (note that for the rest of $b' \neq b, v(b') = h_b(b')$). Such a pair $\langle b, v \rangle$ will be denoted as $b \oplus \xi$. For $t \in \cantortree$, define $t^0 = q(t^\frown \vec{0})$ and $t^1 = q(t^\frown \vec{0})$.

Note that, for such $a^t_n$, and since $s(a) = D(b)(t)$ extends $t$, we have that
\begin{gather*}
  h_b(\langle a^t_n,D(b)(t)^\frown \vec{0}\rangle) =  q(g_b(\langle a^t_n, D(b)(t)^\frown \vec{0}\rangle)) =\\ = q(\Gamma_t^{D(b)(t)}(D(b)(t)^\frown \vec{0})) = t^0.  
\end{gather*}
Similarly, we have that 
$$
h_b(\langle a^t_n,D(b)(t)^\frown \vec{1}\rangle) = t^1.
$$
Now, observe that, for every sequence $\{a^t_n : n \in \omega \}$ as before, we have that $\{a^t_n \} \rightarrow b$ as $n \rightarrow \infty$, and therefore $a^t_n + D(b)(t)^\frown \vec{0} \rightarrow b \oplus t^0$ and $a^t_n + D(b)(t)^\frown \vec{1} \rightarrow b \oplus t^1$.

Using this, (\ref{smldm}), and the lower semi-continuity of $\delta$ we get that 
\begin{equation}\tag{**}\label{smldm2}
    \delta(b \oplus t^0, b \oplus t^1) \leq \frac{1}{4^{|t|}}
\end{equation}

Finally, notice that for every $n \in \omega$, there is a way to order $2^n = \{ t_j : 1 \leq j \leq 2^{|n|} \}$ such that $t_0$ is the sequence of zeroes, $t_{2^{|n|}}$ is the sequence of ones, and for $j < 2^{|n|}, t_j^1 = t_{j+1}^0$. Using this, (\ref{smldm2}), and the triangle inequality, we get that for all $n \in W$,
$$
\delta(b \oplus \vec{0}, b \oplus \vec{1}) \leq \frac{2^n}{4^n} = \frac{1}{2^n}.
$$
Since this happens for infinitely many $n$, we get that $\delta(b \oplus \vec{0}, b \oplus \vec{1}) = 0$, which is the contradiction we were looking for, as $\delta$ is a metric and $\delta(b \oplus \vec{0}, b \oplus \vec{1}) > 0$.
\end{proof}

This finishes the recipe, so all is left is to construct a basic* space of cardinality smaller than $\mathfrak{c}$.
\section{Constructing a small basic* space}

This section is the main part of this work, and it will be solely dedicated to construct a basic* space of small weight, under $\diamondsuit (\non{\M})$. The important part of the construction is to show how the sequence sets $\langle C_b : b \in B \rangle$ and the function $D$ from the condition (\ref{basic}) is obtained. Note that, for fixed $A,B$, the sequence $\langle C_b:b\in B\rangle$ and the function $D$ define a basic* space exactly if the condition (\ref{basic}) is satisfied and $\langle C_b: b\in B\rangle$ is an \emph{almost-disjoint family}, i.e., for any $b\neq d$, the set $C_b\cap C_d$ is finite.

\begin{theorem}
    Under $\diamondsuit (\non{\M})$ there is a basic$^*$ space of size $\aleph_1$.\label{smallbasic}
\end{theorem}
\begin{proof}
Let $B$ be the set of ordinals of $\omega_1$ that are the limit of limit ordinals, and, for each $t \in \cantortree$ let $A_t = \{t\} \times \omega_1$. Fix bijective $\varepsilon_\beta : \omega \rightarrow \lim(\omega_1) \cap \beta, e_\beta:\omega \to \beta$ for each $\beta \in B$ and let $\pi_1, \pi_2$ be the projection in the first and second coordinates, respectively.

First, we are going to construct a suitable function $F$ for $\diamondsuit(\non{\M})$. 

Let $\Gamma$ be the set of triples $(\langle C_\alpha:\alpha<\beta, \alpha \in \lim(\omega_1)\rangle, \langle G_t:t\in \cantortree \rangle, H)$ with $\beta\in B,\, H\in \cantor,\, C_\alpha \subseteq \cantortree\times \alpha,\, G_t\subseteq \{t\} \times \beta$. The \emph{length} of $\gamma \in \Gamma$ is the ordinal $\beta$ as in before.  Given $\gamma\in \Gamma$ of length $\beta$ and $\delta \in B\cap \beta$ let $\gamma\restriction \delta=(\langle C_\alpha:\alpha<\delta, \alpha \in \lim(\omega_1)\rangle, \langle G_t\cap (\cantortree\times \delta) :t\in \cantortree \rangle, H)$. An encoding $E:\Gamma\to 2^{<\omega_1}$ \emph{preserves restriction} if for any $\gamma\in \Gamma$ of length $\beta\in B$ and any $\delta\in B\cap \beta$ we have that $E(\gamma)\in 2^\beta$ and  $E(\gamma\restriction\delta)=E(\gamma)\restriction \delta$.
Such an encoding can be constructed inductively by sending $\langle C_\alpha: \beta\leqslant \alpha < \beta+\omega^2,\alpha\in \lim(\omega_1)\rangle$ and $ \langle G_t\cap (\cantortree\times [\beta,\beta+\omega^2) :t\in \cantortree \rangle$ into $2^{[\beta,\beta+\omega^2)}$. We leave the details of the construction to the reader and from now on we will implicitly use a fixed, restriction-preserving Borel encoding, thus assuming that the domain of $F$ is $\Gamma$.

We will say that $\gamma\in \Gamma$ is \emph{suitable} if for each $\alpha, \alpha' \in \lim(\omega_1)\cap \beta$ with $\alpha \neq \alpha'$ and each $t \in \cantortree$,
\begin{enumerate}[(a)]
    \item $C_\alpha \cap A_t$ is infinite,
    \item $\bigcup_{\alpha < \beta}C_\alpha = \cantortree \times \beta$,
    \item $G_t \subseteq \{t\} \times \beta$ is cofinal in $\beta$,
    \item $C_\alpha \cap C_{\alpha'}$ is finite.
\end{enumerate}

For such suitable sequences, define the sets 
$$
C_n = C_{\varepsilon_\beta (n)} \setminus \bigcup_{k < n} C_{\varepsilon_\beta (k)}.
$$
Note that, for each $t \in \cantortree$, the family $\{ C_n \cap G_t : n \in \omega \}$ is a partition of $G_t$. Then, define $f_t (n) = \min e^{-1}_\beta (\pi_2( C_n \cap G_t))$ whenever is possible (otherwise $n \notin dom(f_t)$). The properties (a),(b) and (c) guarantee that each $f_t$ is an infinite partial function.

Let 
$$
F(\gamma) = \{ \langle x,y \rangle \in \cantor \times \omega^{\,\cantortree \times \omega} : \forall^\infty n \in \omega\,\exists t \in 2^{n}
$$
$$
\text{such that either }H(t) \neq x(t) \text{ or }\forall^\infty \ell \in \mathrm{dom}(f_t)\,( f_t(\ell) \geq y(t,\ell)) \}.
$$

It is routine to check that the following two sets are meager, respectively in $\cantor$ and $\omega^{\,\cantortree \times \omega}$: $J_\gamma=\{x\in \cantor: \forall^\infty n\in\omega\, \exists t\in 2^n\, H(t)\neq x(t)\}$ and $ I_\gamma=\{y\in \omega^{\,\cantortree \times \omega}:\exists t\in \cantortree\,\forall^\infty \ell\in \mathrm{dom}(f_t)\,f_t(\ell)\geqslant y(t,\ell)\}$. Now, if $\langle x,y\rangle \in F(\gamma)$ then $x\in J_\gamma$ or $y\in I_\gamma$, hence $F(\gamma)$ is meager in $\cantor \times \omega^{\cantortree \times \omega}$. 

For a non-suitable sequence $\gamma$ let $F(\gamma)=\emptyset$. Since all the constructions carried out in order to build $F(\gamma)$ were only dependent of the natural numbers, fixed bijections between countable sets, and the elements of the sequence $\gamma$, we have that $F$ is a Borel function.

\begin{remark}
    The important thing to keep in mind about the construction of $F$ is that, if $\gamma = (\langle C_\alpha:\alpha<\beta, \alpha \in \lim(\omega_1)\rangle, \langle G_t:t\in \cantortree \rangle, H)$ is suitable, then $\langle x,y \rangle \notin F(\gamma)$ implies that, for infinitely many levels of $\cantortree$, $x$ and $H$ will agree, and for $t$ in those levels, $y(t)(\ell) > f_t(\ell)$ for infinitely many $\ell \in dom(f_t)$ for the function $f_t$ defined above. 
\end{remark}

Let $\Delta : \omega_1 \rightarrow \cantor \times \omega^{\,\cantortree \times \omega}$ be the sequence given by 
$\diamondsuit(\non{\M})$ for that function $F$. 
Let $D = \pi_1(\Delta)$. Recursively we are going to construct a sequence $\langle C_\alpha : \alpha \in \lim (\omega_1) \rangle$ such that, for each $\alpha \in \lim (\omega_1)$,
\begin{enumerate}[(i)]
    \item $C_\alpha \subseteq \cantortree \times \alpha$,
    \item for each $t \in \cantortree$, $C_\alpha \cap A_t$ is infinite,
    \item for each $\alpha' < \alpha$, $C_{\alpha} \cap C_{\alpha'}$ is finite,
    \item if $\alpha \in B$, $\bigcup_{\alpha' < \alpha}C_{\alpha'} = \cantortree \times \alpha$,
    \item if $\langle C_\alpha' : \alpha' < \alpha \rangle$ forms a suitable triple $\gamma$ with some $\langle G_t : t \in \cantortree \rangle$ and $H \in \cantor$ and if $\Delta(\alpha) \notin F(\gamma)$, then there are infinitely many $n\in \omega$ such that for every $t \in 2^n$ the set $C_\alpha \cap G_t$ is infinite.
\end{enumerate}

For $\beta \in \lim(\omega_1) \setminus B$ let $\alpha$ be the maximal element of $\lim(\omega_1)\cap \beta$ and set $C_\beta=\cantortree\times[\alpha,\beta)$. Assume that $\beta \in B$ and that the set $C_\alpha \subseteq A \cap (\cantortree \times \alpha)$ has been constructed for $\alpha < \beta$ with $\alpha \in \lim (\omega_1)$. Exactly like before, let
$$
C_n = C_{\varepsilon_\beta (n)} \setminus \bigcup_{k < n} C_{\varepsilon_\beta (n)}.
$$
 Let $C_\beta = \{ \langle t,a\rangle \in C_n : n \in \omega, t \in 2^{\leq n} \text{ and } e^{-1}_\beta (a) \leq \pi_2(\Delta(\beta))(t,n) \}$. Note that, for each $n$, $C_\beta$ takes at most $\pi_2(\Delta(\beta))(t,n)$ many elements from $C_n \cap A_t$, and only for finitely many $t$, so therefore, for each $\alpha < \beta$, $C_\alpha \cap C_\beta$ is finite, and therefore the properties (i), (ii), (iii) and (iv) are satisfied.

To see that our construction satisfies property (v), assume that $\Delta(\beta) \notin F(\gamma)$ for some suitable $\gamma$ containing $\langle C_\alpha : \alpha < \beta \rangle$, $\langle G_t : t \in \cantortree \rangle$ and $H \in \cantor$. Then there is an infinite amount of $n\in \omega$ such that for all $t \in \cantortree$, we have that both $H(t) = \pi_1(\Delta(\beta))(t) = D(\beta)(t)$ and $f_t (\ell) < \pi_2(\Delta (\beta))(t)(\ell)$ for infinitely many $\ell \in \mathrm{dom}(f_t)$, where $f_t (\ell) = \min e^{-1}_\beta (\pi_2( C_\ell \cap G_t))$ (see the Remark above). For $t \in \cantortree$ and $\ell \geq |t|$ such that $f_t (\ell) < \pi_2(\Delta (\beta))(t)(\ell)$, we have that at least $e_\beta (\min e^{-1}_\beta (\pi_2( C_\ell \cap G_t))) \in C_\beta \cap G_t$, and, in particular, the set $C_\beta \cap G_t$ is infinite, which is the conclusion we wanted.

The properties (1), (2), (3) and (4) follow easily from (i), (ii) and (iii). We will show that Property (5) is satisfied: Let $s: A \rightarrow \cantortree$. For each $t \in \cantortree$, pick an element $H(t) \in \cantortree$ such that $H(t) = s(\alpha)$ for uncountably many $\alpha \in A_t$. Let $G_t = \{ \alpha \in A_t : H(t) = s(\alpha) \}$. Notice that, for each $t \in \cantortree$, the set of all $\delta \in B$ such that $G_t$ is cofinal in $\delta$ is a club and therefore the set $\mathcal{C} = \{ \delta \in B : (\forall t\in\cantortree) \ G_t\text{ is cofinal in }\delta\}$ is a club.

Consider a triple $\gamma = ( \langle C_\alpha : \alpha \in \lim(\omega_1)\rangle,\langle G_t: {t \in \cantortree}\rangle , H )$. Since $\Delta$ was the function witnessing $\diamondsuit (\non{\M})$, the set $\mathcal{S} = \{ \beta \in \omega_1 : \Delta(\beta) \notin F(\gamma \restriction \beta) \}$ is stationary in $\omega_1$ and therefore $\mathcal{S} \cap \mathcal{C}$ is uncountable. So, for each $\beta \in \mathcal{S} \cap \mathcal{C}$, we have that $\Delta(\beta) \notin F(\gamma \restriction \beta)$ which means that there is an infinite amount of $n \in \omega$ such that for all $t \in 2^n$ we have that $D(\beta)(t) = \pi_1(\Delta(\beta))(t) = H(t)$ and, by (v), the set $C_\beta \cap (G_t \cap \beta)$ is infinite, but the set $G_t$ is exactly the set of  $\alpha\in A_t$ such that $H(t) = s(\alpha)$, so the conclusion follows.
\end{proof}

Combining the above result with Theorems \ref{l0},\ref{l1}, and Lemma \ref{lemmaweightfinalspaces} we get the main theorem of our work:
\begin{theorem}\label{maintheorem}%
    Under $\diamondsuit (\non{\M})$ there is a Radon-Nikod\'{y}m compact space of weight $\aleph_1$, and a continuous function whose image is not Radon-Nikod\'{y}m.
\end{theorem}

There are several models where $\diamondsuit (\non{\M})$ holds and the continuum is large. Adding $\kappa$ Cohen reals for $\kappa$ of uncountable cofinality yields such model. Other well-known models where $\diamondsuit (\non{\M})$ holds are Sacks', Miller's and Silver's model. In fact, in several \emph{canonical} models where $\non{\mathcal{M}} = \aleph_1$, the principle $\diamondsuit (\non{\M})$ also holds. The reader can consult \cite{ParametrizedDiamonds} to see in which canonical models this principle holds.

We would like to finish this work by looking at the limitations of our construction. First, we will need a characterization of $\non{\mathcal{M}}$ due to Bartoszy\'{n}ski.

\begin{theorem}[\cite{Bart}]
  $\non{\M}=\min\{|\A|: \A\subseteq \omega^\omega, \forall f\in \omega^\omega\, \exists g \in \A\, \exists^\infty n\; g(n)=f(n)\}.$  
\end{theorem}

Using this, it is easy to show the following result, presenting a limitation of this method.

\begin{proposition}
    If $K$ is a basic* space, then $w(K) \geq \non{\M}.$
\end{proposition}
\begin{proof}
    Let $D$ be a function witnessing condition (\ref{basic}). Notice that by renaming $\baire$, we can easily see that $\non{\M}=\min\{|\A|: \A\subseteq \cantor, \forall f\in \cantor\, \exists g\in \A\, \exists^\infty n \, \forall t\in 2^n \; g(t)=f(t)\}$. If $f \in \cantor$, we define $s_f:A\to \cantortree$ by $s_f(a)=f(t)$ whenever $a\in A_t$. For some $b\in B$ we have that $D(b)(t)=f(t)$ for each $t\in 2^n$ for infinitely many $n$, i.e., the image of $D$ is a family satisfying the conditions for $\non{\mathcal{M}}$. The conclusion follows from this and from Lemma \ref{lemmaweightK}.
\end{proof}

This shows that our construction has the smallest weight possible, so if one wants to construct one of these spaces in a similar way, it would be necessary to change the condition (\ref{basic}) from the definition of basic* space.

Recall Avil\'{e}s' result from \cite{AvilesRNUnderB}, that all quasi Radon-Nikod\'{y}m spaces are Radon-Nikod\'{y}m. Our spaces have weight $\non{\mathcal{M}}$ and it is known that $\mathfrak{b} \leq \non{\mathcal{M}}$ (see \cite{BlassCardinalCharacteristics} for a proof), which leads to the following question.

\begin{question}
    Is it consistent to have $\non{\mathcal{M}} > \aleph_1$, but have a a Radon-Nikod\'{y}m compact space of weight $\aleph_1$ with a continuous image that is not Radon-Nikod\'{y}m?
\end{question}

One approach to a possible answer to this question would be to, in some way, modify the condition (\ref{basic}), to allow such space to be constructed using $\diamondsuit(\mathfrak{b})$ ($\diamondsuit(\mathfrak{b})$ is the principle $\diamondsuit(\mathbf{B})$, where $\mathbf{B}$ was defined in the introduction).

\begin{acknowledgements}
We would like to thank Piotr Koszmider for introducing us to the problem, and for sharing his valuable ideas related to it. We would also like to thank the participants of the joint topology and set theory seminar from Uniwersytet Wroc\l{}awski and from the Wroc\l{}aw University of Science and Technology for many hours of stimulating conversations.
\end{acknowledgements}


\bibliographystyle{alphadin} 
\bibliography{main}

\begin{thebibliography}{CMHsMA14}


\providecommand{\url}[1]{\texttt{#1}}
\expandafter\ifx\csname urlstyle\endcsname\relax
  \providecommand{\doi}[1]{doi: #1}\else
  \providecommand{\doi}{doi: \begingroup \urlstyle{rm}\Url}\fi

\bibitem[AK10]{SurveyProblem}
\textsc{Avil\'es}, Antonio ; \textsc{Kalenda}, Ond\v rej F.~K.:
\newblock Compactness in {B}anach space theory---selected problems.
\newblock {In: }\emph{Rev. R. Acad. Cienc. Exactas F\'is. Nat. Ser. A Mat. RACSAM} 104 (2010), Nr. 2, 337--352.
\newblock \url{http://dx.doi.org/10.5052/RACSAM.2010.21}. --
\newblock DOI 10.5052/RACSAM.2010.21. --
\newblock ISSN 1578--7303,1579--1505

\bibitem[AK13]{AvilesKoszmider}
\textsc{Avil\'{e}s}, Antonio ; \textsc{Koszmider}, Piotr:
\newblock A continuous image of a {R}adon-{N}ikod\'{y}m compact space which is not {R}adon-{N}ikod\'{y}m.
\newblock {In: }\emph{Duke Math. J.} 162 (2013), Nr. 12, 2285--2299.
\newblock \url{http://dx.doi.org/10.1215/00127094-2348447}. --
\newblock DOI 10.1215/00127094--2348447. --
\newblock ISSN 0012--7094,1547--7398

\bibitem[AK16]{TopicsBanach}
\textsc{Albiac}, Fernando ; \textsc{Kalton}, Nigel~J.:
\newblock \emph{Graduate Texts in Mathematics}. Bd. 233: {\emph{Topics in {B}anach space theory}}.
\newblock Second.
\newblock Springer, [Cham], 2016. --
\newblock  xx+508 S.
\newblock \url{http://dx.doi.org/10.1007/978-3-319-31557-7}.
\newblock \url{http://dx.doi.org/10.1007/978-3-319-31557-7}. --
\newblock ISBN 978--3--319--31555--3; 978--3--319--31557--7. --
\newblock With a foreword by Gilles Godefory

\bibitem[Arh96]{GenTopII}
\textsc{Arhangel'skii}, A.~V. (Hrsg.):
\newblock \emph{Encyclopaedia of Mathematical Sciences}. Bd.~50: {\emph{General topology. {II}}}.
\newblock Springer-Verlag, Berlin, 1996. --
\newblock  vi+256 S.
\newblock \url{http://dx.doi.org/10.1007/978-3-642-77030-2}.
\newblock \url{http://dx.doi.org/10.1007/978-3-642-77030-2}. --
\newblock ISBN 3--540--54695--2. --
\newblock Compactness, homologies of general spaces, A translation of {\it Current problems in mathematics. Fundamental directions. Vol.\ 50} (Russian), Akad.\ Nauk SSSR, Vsesoyuz.\ Inst.\ Nauchn.\ i Tekhn.\ Inform., Moscow, 1989 [MR1004024 (90a:54002)], Translation by J. M. Lysko

\bibitem[Arv02]{ArvanitakisRNComp}
\textsc{Arvanitakis}, Alexander~D.:
\newblock Some remarks on {R}adon-{N}ikod\'{y}m compact spaces.
\newblock {In: }\emph{Fund. Math.} 172 (2002), Nr. 1, 41--60.
\newblock \url{http://dx.doi.org/10.4064/fm172-1-4}. --
\newblock DOI 10.4064/fm172--1--4. --
\newblock ISSN 0016--2736,1730--6329

\bibitem[Avi05]{AvilesRNUnderB}
\textsc{Avil\'{e}s}, Antonio:
\newblock Radon-{N}ikod\'{y}m compact spaces of low weight and {B}anach spaces.
\newblock {In: }\emph{Studia Math.} 166 (2005), Nr. 1, 71--82.
\newblock \url{http://dx.doi.org/10.4064/sm166-1-5}. --
\newblock DOI 10.4064/sm166--1--5. --
\newblock ISSN 0039--3223,1730--6337

\bibitem[Bar87]{Bart}
\textsc{Bartoszy\'nski}, Tomek:
\newblock Combinatorial aspects of measure and category.
\newblock {In: }\emph{Fund. Math.} 127 (1987), Nr. 3, 225--239.
\newblock \url{http://dx.doi.org/10.4064/fm-127-3-225-239}. --
\newblock DOI 10.4064/fm--127--3--225--239. --
\newblock ISSN 0016--2736,1730--6329

\bibitem[BJ95]{BarJu}
\textsc{Bartoszy\'nski}, Tomek ; \textsc{Judah}, Haim:
\newblock \emph{Set theory}.
\newblock A K Peters, Ltd., Wellesley, MA, 1995. --
\newblock  xii+546 S. --
\newblock ISBN 1--56881--044--X. --
\newblock On the structure of the real line

\bibitem[Bla10]{BlassCardinalCharacteristics}
\textsc{Blass}, Andreas:
\newblock Combinatorial cardinal characteristics of the continuum.
\newblock \,Version:\,2010.
\newblock \url{http://dx.doi.org/10.1007/978-1-4020-5764-9\_7}.
\newblock {In: }\emph{Handbook of set theory. {V}ols. 1, 2, 3}.
\newblock Springer, Dordrecht, 2010. --
\newblock DOI 10.1007/978--1--4020--5764--9\_7. --
\newblock ISBN 978--1--4020--4843--2, 395--489

\bibitem[CMHsMA14]{irresolvablespace}
\textsc{Cancino-Manr\'iquez}, Jonathan ; \textsc{Hru\v~s\'ak}, Michael  ; \textsc{Meza-Alc\'antara}, David:
\newblock Countable irresolvable spaces and cardinal invariants.
\newblock {In: }\emph{Topology Proc.} 44 (2014), S. 189--196. --
\newblock ISSN 0146--4124,2331--1290

\bibitem[Fab06]{OverclassesFabian}
\textsc{Fabian}, Mari\'an:
\newblock Overclasses of the class of {R}adon-{N}ikod\'ym compact spaces.
\newblock \,Version:\,2006.
\newblock \url{http://dx.doi.org/10.1017/CBO9780511721366.011}.
\newblock {In: }\emph{Methods in {B}anach space theory} Bd. 337.
\newblock Cambridge Univ. Press, Cambridge, 2006. --
\newblock DOI 10.1017/CBO9780511721366.011. --
\newblock ISBN 978--0--521--68568--9; 0--521--68568--0, 197--214

\bibitem[GAHHHs18]{ScarStone}
\textsc{Gaspar-Arreola}, Miguel~\. ; \textsc{Hern\'andez-Hern\'andez}, Fernando  ; \textsc{Hru\v~s\'ak}, Michael:
\newblock Scattered spaces from weak diamonds.
\newblock {In: }\emph{Israel J. Math.} 225 (2018), Nr. 1, 427--449.
\newblock \url{http://dx.doi.org/10.1007/s11856-018-1669-1}. --
\newblock DOI 10.1007/s11856--018--1669--1. --
\newblock ISSN 0021--2172,1565--8511

\bibitem[Jec03]{STJech}
\textsc{Jech}, Thomas:
\newblock \emph{Set theory}.
\newblock millennium.
\newblock Springer-Verlag, Berlin, 2003 (Springer Monographs in Mathematics). --
\newblock  xiv+769 S. --
\newblock ISBN 3--540--44085--2

\bibitem[MDS10]{CoveringProperties}
\textsc{Morgan}, C. ; \textsc{Da~Silva}, S.~G.:
\newblock Covering properties which, under weak diamond principles, constrain the extents of separable spaces.
\newblock {In: }\emph{Acta Math. Hungar.} 128 (2010), Nr. 4, 358--368.
\newblock \url{http://dx.doi.org/10.1007/s10474-010-9210-y}. --
\newblock DOI 10.1007/s10474--010--9210--y. --
\newblock ISSN 0236--5294,1588--2632

\bibitem[MHD04]{ParametrizedDiamonds}
\textsc{Moore}, Justin~T. ; \textsc{Hru\v{s}\'{a}k}, Michael  ; \textsc{D\v{z}amonja}, Mirna:
\newblock Parametrized {$\diamondsuit$} principles.
\newblock {In: }\emph{Trans. Amer. Math. Soc.} 356 (2004), Nr. 6, 2281--2306.
\newblock \url{http://dx.doi.org/10.1090/S0002-9947-03-03446-9}. --
\newblock DOI 10.1090/S0002--9947--03--03446--9. --
\newblock ISSN 0002--9947,1088--6850

\bibitem[Nam02]{GenerRNSpaces}
\textsc{Namioka}, I.:
\newblock On generalizations of {R}adon-{N}ikod\'ym compact spaces.
\newblock {In: }\emph{Proceedings of the 16th {S}ummer {C}onference on {G}eneral {T}opology and its {A}pplications ({N}ew {Y}ork)} Bd.~26, 2001/02. --
\newblock ISSN 0146--4124,2331--1290, S. 741--750

\bibitem[Nam87]{NamiokaRNFragment}
\textsc{Namioka}, I.:
\newblock Radon-{N}ikod\'ym compact spaces and fragmentability.
\newblock {In: }\emph{Mathematika} 34 (1987), Nr. 2, 258--281.
\newblock \url{http://dx.doi.org/10.1112/S0025579300013504}. --
\newblock DOI 10.1112/S0025579300013504. --
\newblock ISSN 0025--5793

\bibitem[Rey82]{GSGReynov}
\textsc{Reynov}, O.~I.:
\newblock On a class of {H}ausdorff compacts and {GSG} {B}anach spaces.
\newblock {In: }\emph{Studia Math.} 71 (1981/82), Nr. 2, 113--126.
\newblock \url{http://dx.doi.org/10.4064/sm-71-2-113-126}. --
\newblock DOI 10.4064/sm--71--2--113--126. --
\newblock ISSN 0039--3223,1730--6337

\bibitem[RST24]{DowkerSpace}
\textsc{Rinot}, Assaf ; \textsc{Shalev}, Roy  ; \textsc{Todorcevic}, Stevo:
\newblock A new small {D}owker space.
\newblock {In: }\emph{Period. Math. Hungar.} 88 (2024), Nr. 1, 102--117.
\newblock \url{http://dx.doi.org/10.1007/s10998-023-00541-6}. --
\newblock DOI 10.1007/s10998--023--00541--6. --
\newblock ISSN 0031--5303,1588--2829

\bibitem[Voj92]{InvariantsGaloisTukey}
\textsc{Vojt\'{a}\v{s}}, Peter:
\newblock Topological cardinal invariants and the {G}alois-{T}ukey category.
\newblock {In: }\emph{Recent developments of general topology and its applications ({B}erlin, 1992)} Bd.~67.
\newblock Akademie-Verlag, Berlin, 1992. --
\newblock ISBN 3--05--501423--5, S. 309--314

\end{thebibliography}

\end{document}